\documentclass[english, 10pt]{article}

\usepackage[margin=1in]{geometry}
\usepackage{hyperref}
\usepackage{amsthm,amsfonts,amssymb,amsmath}
\usepackage{graphicx}
\usepackage{xcolor}

\newtheorem{theo}{Theorem}[section]
\newtheorem{propo}[theo]{Proposition}
\newtheorem{lem}[theo]{Lemma}

\newtheorem{claim}[theo]{Claim}

\newtheorem{quest}[theo]{Question}

\theoremstyle{definition}

\newtheorem{remark}[theo]{Remark}

\numberwithin{equation}{section}
\numberwithin{figure}{section}

\usepackage{babel}
\usepackage[T1]{fontenc}
\usepackage[utf8]{inputenc}

\newcommand{\eps}{\varepsilon}
\newcommand{\su}{\subseteq}
\newcommand{\mi}{\setminus}
\newcommand{\Fp}{\mathbb{F}_p}
\newcommand{\Fpn}{\mathbb{F}_p^{n}}

\title{A polynomial bound for the arithmetic $k$-cycle removal lemma in vector spaces}
\author{Jacob Fox\thanks{Department of Mathematics, Stanford University, Stanford, CA 94305. Email: \texttt{jacobfox@stanford.edu}. Research supported by a Packard Fellowship and by NSF Career Award DMS-1352121.}
\and L\'{a}szl\'{o} Mikl\'{o}s Lov\'{a}sz\thanks{Department of Mathematics, UCLA, Los Angeles, CA 90095-1555. Email \texttt{lmlovasz@math.ucla.edu}. Research supported by NSF Postdoctoral Fellowship Award DMS-1705204.}
\and Lisa Sauermann\thanks{Department of Mathematics, Stanford University, Stanford, CA 94305. Email: \texttt{lsauerma@stanford.edu}. Research supported by Jacob Fox's Packard Fellowship.}}

\begin{document}

\maketitle

\begin{abstract}
	For each $k\geq 3$, Green proved an arithmetic $k$-cycle removal lemma 
	for any abelian group $G$. The best known bounds relating the parameters 
	in the lemma for general $G$ are of tower-type.  For $k>3$, even in the 
	case $G=\mathbb{F}_2^n$ no better bounds were known  prior to this paper. 
	This special case has received considerable attention due to its close 
	connection to property testing of boolean functions. For every $k\geq 3$, 
	we prove a polynomial bound relating the parameters for $G=\mathbb{F}_p^n$, 
	where $p$ is any fixed prime. This extends the result for $k=3$ by the first two authors. Due to substantial issues with generalizing the proof of the $k=3$ case, a new strategy is developed in order to prove the result for $k>3$. 
\end{abstract}


\section{Introduction}

Motivated by removal lemmas in graph theory, Green \cite{green} proved the following arithmetic removal lemma for abelian groups:

\begin{theo}[\cite{green}]\label{theo-green} For $k\geq 3$ and any $0<\eps<1$ there exists $\delta=\delta(k,\eps)>0$ such that for any finite abelian group $G$ and any $X_1,\dots,X_k\su G$ at least one of the following holds: the number of $k$-tuples $(x_1,\dots,x_k)\in X_1\times \dots\times X_k$ satisfying $x_1+\dots+x_k=0$ is at least $\delta \vert G\vert^{k-1}$, or we can delete less than $\eps \vert G\vert$ elements from each of the sets $X_1$,\dots, $X_k$ such that afterwards no such $k$-tuples remain.
\end{theo}

Green's proof relies on an arithmetic regularity lemma based on Fourier analysis, and his lower bound for $\delta$ is of tower-type ($1/\delta$ is bounded from above by a tower of twos of height polynomial in $k$ and in $1/\eps$). Kr\'{a}l, Serra, and Vena \cite{kral-serra-vena} found an alternative proof, deducing Theorem \ref{theo-green} from the $k$-cycle removal lemma in graphs. Their proof generalizes Theorem \ref{theo-green} to all finite groups (not necessarily abelian). However, relying on the current best known bound for the $k$-cycle removal lemma in graphs, the lower bound on $\delta$ in Theorem \ref{theo-green} obtained from the Kr\'al-Serra-Vena proof is still of tower-type (with the tower height logarithmic in $1/\eps$, using the first author's bound for the graph case in \cite{fox}).

The problem of improving the bounds for $\delta$ in Theorem \ref{theo-green} has received considerable attention (see \cite{BHATTACHARYYA13,BCSX10,BGRS12,BGS15, bx, fox,green,HX15}). The case of $G=\mathbb{F}_2^n$ has attracted particular interest, since it is closely connected to property testing of boolean functions.

In this paper, we will consider the case $G=\Fpn$, where $p\geq 2$ is a prime fixed throughout, while $n$ remains arbitrary. For convenience, set $N=\vert \Fpn\vert= p^n$.

For $G=\Fpn$ the first two authors \cite{foxlovasz} proved Green's arithmetic removal lemma for $k=3$ (i.e. the arithmetic triangle removal lemma) with a polynomial bound on $\delta$ (while $p$ is fixed):

\begin{theo}[\cite{foxlovasz}]\label{theok=3} Let $0<\eps<1$ and $\delta=\eps^{C_{p,3}}$. Then, for any $X_1,X_2,X_3\su \Fpn$, at least one of the following holds: the number of triples $(x_1,x_2,x_3)\in X_1\times X_2\times X_k$ satisfying $x_1+x_2+x_3=0$ is at least $\delta N^{2}$, or we can delete less than $\eps N$ elements from each of the sets $X_1$, $X_2$ and $X_3$ such that afterwards no such triples remain.
\end{theo}

Here, $C_{p,3}$ is a constant just depending on $p$ and it is given by $C_{p,3}=1+\frac{1}{c_{p,3}}$, where $0<c_{p,3}<1$ is defined via
$$\inf_{0<t<1}\frac{t^0+t^1+\dots+t^{p-1}}{t^{(p-1)/3}}=p^{1-c_{p,3}}.$$
It is not difficult to show that $C_{p,3}=\Theta(\log p)$, see \cite[p. 20]{blasiak-et-al}. It was also shown in \cite{foxlovasz}, based on arguments in \cite{bx} and \cite{KSS17}, that this constant $C_{p,3}$ is the smallest possible exponent for which Theorem \ref{theok=3} is true.
Note that Theorem \ref{theok=3} differs slightly from \cite[Theorem 1]{foxlovasz}, but can be easily obtained from \cite[Theorem 3]{foxlovasz}.

The goal of this paper is to prove that, in the case where $G=\Fpn$ for a fixed prime $p\geq 2$, there is a polynomial bound on $\delta$ in Theorem \ref{theo-green} for any fixed $k\geq 3$.

Given sets $X_1,\dots,X_k\su \Fpn$, a \emph{$k$-cycle} is a $k$-tuple $(x_1,\dots,x_k)\in X_1\times \dots\times X_k$ with $x_1+\dots+x_k=0$. With this notation, our main result is the following, where $C_{p,k}$ is a constant that only depends on $p$ and $k$.

\begin{theo}\label{theo1} Let $k\geq 3$, $0<\eps<1$ and $\delta=\eps^{C_{p,k}}$. Then, for any $X_1,\dots,X_k\su \Fpn$, at least one of the following holds: the number of $k$-cycles $(x_1,\dots,x_k)\in X_1\times \dots\times X_k$ is at least $\delta N^{k-1}$, or we can delete less than $\eps N$ elements from each of the sets $X_1$,\dots, $X_k$ so that afterwards no $k$-cycles remain.
\end{theo}

For the exponent of $\eps$ in the polynomial bound for $\delta$ we will have
$$C_{p,k}=(k-2)\cdot (C_{p,3}-1)+1=\frac{k-2}{c_{p,3}}+1.$$
Here $C_{p,3}$ is the exponent from Theorem \ref{theok=3} defined above. 

It remains an interesting question to determine the smallest possible exponent in Theorem \ref{theo1}.

\begin{quest} What is the smallest possible exponent $C_{p,k}$ for which Theorem \ref{theo1} is true?
\end{quest}

Similarly to \cite[Section 4.2]{bx} (see also \cite[Section 3]{foxlovasz}), one can obtain lower bounds for the exponent from (large) constructions of $k$-colored sum-free sets. Therefore, by the result of \cite{lovaszsauermann}, the smallest possible exponent in Theorem \ref{theo1} must be at least $\frac{k-2}{c_{p,k}}+1$, where $0<c_{p,k}<1$ is defined via
$$\inf_{0<t<1}\frac{t^0+t^1+\dots+t^{p-1}}{t^{(p-1)/k}}=p^{1-c_{p,k}}.$$
Our exponent $C_{p,k}=\frac{k-2}{c_{p,3}}+1$ in Theorem \ref{theo1} is within a factor of $O(\log p)$ from this lower bound (independently of the value of $k$), but it is unclear what the best possible exponent is.

Let us briefly explain why Theorem \ref{theo1} is not a straightforward generalization of Theorem \ref{theok=3}. The proof of Theorem \ref{theok=3} in \cite{foxlovasz} relies on a subspace sampling argument to reduce to the so-called tri-colored sum-free theorem, which has been proved in \cite{blasiak-et-al} following the polynomial method breakthrough of Croot-Lev-Pach \cite{crootlevpach} and subsequently Ellenberg-Gijswijt \cite{ellengijs}. Using Tao's slice rank method \cite{tao}, the tri-colored sum-free theorem easily generalizes to a $k$-colored sum-free theorem for $k>3$. However, the subspace sampling argument in the proof of Theorem \ref{theok=3} in \cite{foxlovasz} cannot be extended to $k>3$ in a direct way. This is because although two different 3-cycles can share at most one point, for $k>3$ two different $k$-cycles can have a larger intersection. These larger intersections drastically influence certain conditional probabilities in the subspace sampling argument in \cite{foxlovasz} in such a way that the proof does not extend to $k>3$. Furthermore, for $k>3$ there are linear dependences between different $k$-cycles that cannot be circumvented by the methods in \cite{foxlovasz} and these lead to additional issues with the subspace sampling argument.

Due to these difficulties, instead of trying to apply the subspace sampling method from \cite{foxlovasz}, we will use a very different strategy. However, our proof relies on the results in \cite{foxlovasz} for the case $k=3$ to start an induction on $k$.

The main part of our proof is Proposition \ref{propo3} below. Although a similar statement for the special case $k=3$ occurs in \cite{foxlovasz}, our proof of Proposition \ref{propo3} uses a completely different strategy than the subspace sampling method in \cite{foxlovasz}. We will now outline how to deduce Theorem \ref{theo1} from Proposition \ref{propo3}, which is very similar to the deduction in \cite{foxlovasz}.

Let us call two $k$-cycles $(x_1,\dots,x_k)$ and $(x_1',\dots,x_k')$ disjoint if $x_i\neq x_i'$ for $i=1,\dots,k$. We will prove the following roughly equivalent version of Theorem \ref{theo1} and then deduce the actual statement of Theorem \ref{theo1} from it at the end of the paper.

\begin{theo}\label{theo2} Let $k\geq 3$, $0<\eps<1$ and $\delta=\eps^{C_{p,k}}$. Let $X_1,\dots,X_k\su \Fpn$ be such that there is a collection of at least $\eps N$ disjoint $k$-cycles in $X_1\times \dots\times X_k$. Then the total number of $k$-cycles $(x_1,\dots,x_k)\in X_1\times \dots\times X_k$ is at least $\delta N^{k-1}$.
\end{theo}

We will prove Theorem \ref{theo2} by induction on $k$. The base case $k=3$ is an alternative version of Theorem \ref{theok=3} above that was also proved by the first two authors \cite[Theorem 3]{foxlovasz}. For the induction step, the main challenge is to prove the following key proposition.

\begin{propo}\label{propo3} Let $k\geq 4$ be given such that Theorem \ref{theo2} is true for all smaller values of $k$. Let $X_1,\dots,X_k\su \Fpn$ and assume that the number of $k$-cycles $(x_1,\dots,x_k)\in X_1\times \dots\times X_k$ is equal to $\delta' N^{k-1}$ for some $\delta'>0$. Finally let $\theta\geq 1$ be a real number such that for every $i=1,\dots,k$, each point of $X_i$ occurs as $x_i$ in at most $\theta\delta'N^{k-2}$ different $k$-cycles $(x_1,\dots,x_k)\in X_1\times \dots\times X_k$. Then
$$\delta'\theta^{C_{p,k}}\geq 2^{-4kC_{p,k}}.$$
\end{propo}

In order to perform the induction step for proving Theorem \ref{theo2}, we will use Proposition \ref{propo3}, after repeatedly deleting points which are in a relatively large number of $k$-cycles. This way, we will be able to apply the proposition, taking $\theta$ to be roughly $1/\eps$ (up to logarithmic factors). In this way, we obtain a lower bound on the number of $k$-cycles in $X_1\times \dots\times X_k$. This bound will be similar to the desired bound in Theorem \ref{theo2}. Using a power trick, we can then obtain the actual desired bound, and finish the induction step.

The proof of Proposition \ref{propo3}, which is the heart of our argument, will be given in Section \ref{sect2}, apart from a lemma which we will postpone to Section \ref{sect4}. In Section \ref{sect5} we will perform the induction for proving Theorem \ref{theo2} and finally deduce Theorem \ref{theo1}.

\textit{Notation.} All logarithms are with base 2. For a positive integer $m$, let $[m]=\lbrace 1,\dots,m\rbrace$.

\section{Proof of Proposition \ref{propo3}}\label{sect2}

Let $X_1,\dots,X_k\su \Fpn$ and let the real numbers $\delta'>0$ and $\theta\geq 1$ be as in Proposition \ref{propo3}. Let $\alpha=(\theta\delta')^{\frac{1}{k-2}}$, that is $\theta\delta'=\alpha^{k-2}$. Recall that $k\geq 4$ and that we assume that Theorem \ref{theo2} is true for all smaller values of $k$.

For any subset $I\su [k]$ with size $1\leq \vert I\vert\leq k-2$, let an \emph{$I$-tuple} be a tuple $(x_i)_{i\in I}\in \prod_{i\in I}X_i$, that is, some tuple of elements of $\Fpn$ indexed by the set $I$ where $x_i\in X_i$ for each $i\in I$. Let us call an $I$-tuple $(x_i)_{i\in I}\in \prod_{i\in I}X_i$ \emph{bad} if there are at least $2\alpha^{k-\vert I\vert-1}N^{k-\vert I\vert-1}$ different ways to extend $(x_i)_{i\in I}$ to a $k$-cycle $(x_1,\dots,x_k)\in X_1\times \dots\times X_k$. We say that a $k$-cycle $(x_1,\dots, x_k)\in X_1\times\dots\times X_k$ \emph{contains a bad tuple} if there is $I\su [k]$ such that $(x_i)_{i\in I}$ is a bad $I$-tuple.

Note that in the case $\vert I\vert=1$ we have
$$2\alpha^{k-\vert I\vert-1}N^{k-\vert I\vert-1}=2\alpha^{k-2}N^{k-2}=2\theta\delta'N^{k-2}.$$
So if $\vert I\vert=1$, then there are no bad $I$-tuples, because for each $i=1,\dots,k$, each point of $X_i$ occurs as $x_i$ in at most $\theta\delta'N^{k-2}$ different $k$-cycles $(x_1,\dots,x_k)\in X_1\times \dots\times X_k$.

A crucial step for proving Proposition \ref{propo3} is to show the following lemma.

\begin{lem}\label{newlemma} Let $M$ be a collection of some $k$-cycles in $X_1\times \dots\times X_k$ such that $\vert M\vert\geq r\delta'N^{k-1}$ for a real number $0<r<1$. Furthermore, let $I_1,I_2\su [k]$ be two disjoint subsets with $[k]=I_1\cup I_2$, $2\leq \vert I_1\vert\leq k-2$ and $2\leq \vert I_2\vert\leq k-2$.
Assume that
\begin{itemize}
\item[(i)] For every $I_1$-tuple $(x_i)_{i\in I_1}\in \prod_{i\in I_1}X_i$, there are at most $2\alpha^{k-\vert I_1\vert-1}N^{k-\vert I_1\vert-1}$ different ways to extend $(x_i)_{i\in I_1}$ to a $k$-cycle $(x_1,\dots,x_k)\in M$.
\item[(ii)] For every $(x_1,\dots,x_k)\in M$, the $I_2$-tuple $(x_i)_{i\in I_2}$ is not a bad $I_2$-tuple.
\end{itemize}
Then
$$\delta'\theta^{C_{p,k}}\geq \left(\frac{r}{4k}\right)^{C_{p,k}}.$$
\end{lem}

We will postpone the proof of Lemma \ref{newlemma} to Section \ref{sect4}. In the proof of the lemma, we will use the hypothesis that Theorem \ref{theo2} holds for all smaller values of $k$.

Our strategy for proving Proposition \ref{propo3} is to construct a suitable collection $M$ of $k$-cycles in $X_1\times \dots\times X_k$ to which we can apply Lemma \ref{newlemma}. We will distinguish two cases, whether at most half or more than half of all $k$-cycles in $X_1\times \dots\times X_k$ contain a bad tuple. The first case is relatively easy, because we can apply Lemma \ref{newlemma} to the collection of $k$-cycles not containing any bad tuple (and these will be at least half of all $k$-cycles in $X_1\times \dots\times X_k$). In the second case, the construction of a suitable collection $M$ will be more involved.

First, suppose that the number of $k$-cycles not containing any bad tuple is at least $\frac{1}{2}\delta'N^{k-1}$. Then we can apply Lemma \ref{newlemma} with $r=\frac{1}{2}$ and $M$ being the collection of all $k$-cycles not containing any bad tuple. Indeed, choose any disjoint subsets $I_1,I_2\su [k]$ with $[k]=I_1\cup I_2$, $2\leq \vert I_1\vert\leq k-2$ and $2\leq \vert I_2\vert\leq k-2$, for example $I_1=\lbrace 1,2\rbrace$ and $I_2=\lbrace3,\dots,k\rbrace$ (recall that $k\geq 4$). Let us check the assumptions (i) and (ii):
\begin{itemize}
\item[(i)] Let $(x_i)_{i\in I_1}\in \prod_{i\in I_1}X_i$. If the $I_1$-tuple $(x_i)_{i\in I_1}$ is not bad, then it can be extended in at most $2\alpha^{k-\vert I_1\vert-1}N^{k-\vert I_1\vert-1}$ different ways to a $k$-cycle $(x_1,\dots,x_k)\in X_1\times \dots\times X_k$ and so in particular in at most $2\alpha^{k-\vert I_1\vert-1}N^{k-\vert I_1\vert-1}$ different ways to a $k$-cycle $(x_1,\dots,x_k)\in M$. And if the $I_1$-tuple $(x_i)_{i\in I_1}$ is bad, then it cannot be extended to any $k$-cycle in $M$ at all (recall that we chose $M$ to be the collection of all $k$-cycles not containing a bad tuple). So assumption (i) is satisfied.
\item [(ii)] Let $(x_1,\dots,x_k)\in M$. By the choice of $M$, the $k$-cycle $(x_1,\dots,x_k)$ does not contain any bad tuple. In particular, the $I_2$-tuple $(x_i)_{i\in I_2}$ is not bad.
\end{itemize}
So we can indeed apply Lemma \ref{newlemma} and obtain (as $k\geq 4$)
$$\delta'\theta^{C_{p,k}}\geq \left(\frac{r}{4k}\right)^{C_{p,k}}= \left(\frac{1}{8k}\right)^{C_{p,k}}\geq\left(\frac{1}{8\cdot 2^k}\right)^{C_{p,k}}= \left(\frac{1}{2^{k+3}}\right)^{C_{p,k}}\geq \left(\frac{1}{2^{4k}}\right)^{C_{p,k}}=2^{-4kC_{p,k}}.$$
This would establish the claim of Proposition \ref{propo3}.

So from now on we can assume that the number of $k$-cycles not containing any bad tuple is at most $\frac{1}{2}\delta'N^{k-1}$. Then the number of $k$-cycles that contain a bad tuple is at least
$$\left(1-\frac{1}{2}\right)\delta'N^{k-1}=\frac{1}{2}\delta'N^{k-1}.$$
For each of these $k$-cycles $(x_1,\dots,x_k)$ choose some minimum size set $I\su [k]$ such that $(x_i)_{i\in I}$ is a bad $I$-tuple. Note that by the pigeonhole principle, some set $I\su [k]$ must have been chosen at least $\frac{1}{2^{k+1}}\delta'N^{k-1}$ times. From now on, let us fix such a set $I\su [k]$.

Note that $\vert I\vert\leq k-2$ (since this was assumed in the definition of bad $I$-tuple above) and also $\vert I\vert\geq 2$ since there are no bad tuples if $\vert I\vert =1$.

Let $\ell=\vert I\vert$, then $2\leq \ell\leq k-2$. Upon relabeling the indices, we can assume without loss of generality that $I=[\ell]$. Let
\begin{multline*}
M^{0}=\lbrace (x_1,\dots,x_k)\in X_1\times\dots\times X_k\mid x_1+\dots+x_k=0, \ (x_1,\dots,x_\ell)\text{ is a bad $[\ell]$-tuple},\\
(x_1,\dots,x_{\ell-1})\text{ is not a bad $[\ell-1]$-tuple}\rbrace.
\end{multline*}
By the choice of $I$, the number of $k$-cycles $(x_1,\dots,x_k)$ for which we have chosen $I=[\ell]$ is at least $\frac{1}{2^{k+1}}\delta'N^{k-1}$. For all of these $k$-cycles the $[\ell]$-tuple $(x_1,\dots,x_\ell)$ is bad (because $I=[\ell]$ has been chosen), but the $[\ell-1]$-tuple $(x_1,\dots,x_{\ell-1})$ is not bad (because we chose a minimum size $I$). Hence all these $k$-cycles belong to $M^{0}$, and in particular
$$\vert M^{0}\vert\geq \frac{1}{2^{k+1}}\delta'N^{k-1}.$$

For $j=\ell+2,\dots,k$, let
$$M_j=\lbrace (x_1,\dots,x_k)\in M^{0}\mid (x_1,\dots,x_{\ell-1},x_j) \text{ is not a bad $([\ell-1]\cup\lbrace j\rbrace)$-tuple}\rbrace.$$
We want to apply Lemma \ref{newlemma} to one of these sets $M_j$. First, let us show that some $M_j$ is sufficiently large. For this, let
$$M'=M^{0}\mi(M_{\ell+2}\cup\dots\cup M_k).$$
That is, $M'$ is the collection of all those $(x_1,\dots,x_k)\in X_1\times\dots\times X_k$ with $x_1+\dots+x_k=0$ and such that $(x_1,\dots,x_\ell)$ is a bad $[\ell]$-tuple,  $(x_1,\dots,x_{\ell-1})$ is not a bad $[\ell-1]$-tuple, and
$(x_1,\dots,x_{\ell-1},x_j)$ is a bad $([\ell-1]\cup\lbrace j\rbrace)$-tuple for all $j=\ell+2,\dots,k$. The following lemma states that $M'$ has at most half the size of $M^{0}$, and from this we will conclude that one of the sets $M_j$ must be sufficiently large.

\begin{lem}\label{mprime} We have $\vert M'\vert\leq \frac{1}{2}\vert M^{0}\vert$.
\end{lem}
\begin{proof} Let us partition the elements $(x_1,\dots,x_k)$ of $M^{0}$ into classes according to the $[\ell]$-tuple $(x_1,\dots,x_\ell)$, so in each class all elements agree in the first $\ell$ entries. Note that by the definition of $M^{0}$, for any such class, these first $\ell$ entries $(x_1,\dots,x_\ell)$ form a bad $[\ell]$-tuple, but $(x_1,\dots,x_{\ell-1})$ is not a bad $[\ell-1]$-tuple. Since $(x_1,\dots,x_\ell)$ is a bad $[\ell]$-tuple, it can be extended to at least $2\alpha^{k-\ell-1}N^{k-\ell-1}$ different $k$-cycles $(x_1,\dots,x_k)$, all of which are elements of $M^{0}$. Hence every (non-empty) class contains at least $2\alpha^{k-\ell-1}N^{k-\ell-1}$ elements of $M^0$.

On the other hand, we can establish an upper bound for the number of elements of $M'$ in each class. For this, fix any bad $[\ell]$-tuple $(x_1,\dots,x_\ell)$, such that the $[\ell-1]$-tuple $(x_1,\dots,x_{\ell-1})$ is not bad. Then $(x_1,\dots,x_{\ell-1})$ can be extended to at most $2\alpha^{k-\ell}N^{k-\ell}$ different $k$-cycles $(x_1,\dots,x_k)\in X_1\times\dots\times X_k$.

For the moment, fix any $j=\ell+2,\dots,k$. For each choice for $x_j\in X_j$ such that $(x_1,\dots,x_{\ell-1},x_j)$ is a bad $([\ell-1]\cup\lbrace j\rbrace)$-tuple, the $([\ell-1]\cup\lbrace j\rbrace)$-tuple $(x_1,\dots,x_{\ell-1},x_j)$ can be extended in at least $2\alpha^{k-\ell-1}N^{k-\ell-1}$ different ways to a $k$-cycle $(x_1,\dots,x_k)\in X_1\times\dots\times X_k$. Each of these is also an extension of the $[\ell-1]$-tuple $(x_1,\dots,x_{\ell-1})$. Since the total number of $k$-cycles extending $(x_1,\dots,x_{\ell-1})$ is at most $2\alpha^{k-\ell}N^{k-\ell}$, this implies that there can be at most
$$\frac{2\alpha^{k-\ell}N^{k-\ell}}{2\alpha^{k-\ell-1}N^{k-\ell-1}}=\alpha N$$
choices for $x_j\in X_j$ such that $(x_1,\dots,x_{\ell-1},x_j)$ is a bad $([\ell-1]\cup\lbrace j\rbrace)$-tuple.

To summarize, if we are given a bad $[\ell]$-tuple $(x_1,\dots,x_\ell)$ with the property that the $[\ell-1]$-tuple $(x_1,\dots,x_{\ell-1})$ is not bad, then for each $j=\ell+2,\dots,k$ there are at most $\alpha N$ choices for $x_j\in X_j$ such that $(x_1,\dots,x_{\ell-1},x_j)$ is a bad $([\ell-1]\cup\lbrace j\rbrace)$-tuple. In particular, there are at most $(\alpha N)^{k-\ell-1}=\alpha^{k-\ell-1}N^{k-\ell-1}$ ways to extend $(x_1,\dots,x_\ell)$ to an element of $M'$ (because after choosing all $x_j\in X_j$ for $j=\ell+2,\dots,k$ there is at most one choice for the remaining element $x_{\ell+1}$ as we need $x_1+\dots+x_k=0$). Thus, each of the partition classes considered above contains at most $\alpha^{k-\ell-1}N^{k-\ell-1}$ elements of $M'$.

All in all, we have partitioned $M^0$ in such a way that each partition class contains at least
$2\alpha^{k-\ell-1}N^{k-\ell-1}$ elements of $M^0$, but at most $\alpha^{k-\ell-1}N^{k-\ell-1}$ elements of $M'$. Thus, $\vert M'\vert\leq \frac{1}{2}\vert M^0\vert$ as desired.
\end{proof}

By Lemma \ref{mprime} we have
$$\vert M_{\ell+2}\cup\dots\cup M_k\vert=\vert M^0\vert-\vert M'\vert\geq \frac{1}{2}\vert M^0\vert\geq \frac{1}{2^{k+2}}\delta'N^{k-1}.$$
Hence there is some $j\in \lbrace \ell+2,\dots,k\rbrace$ with $\vert M_j\vert\geq \frac{1}{k2^{k+2}}\delta'N^{k-1}$. Upon relabeling the indices $\ell+2,\dots,k$ we can assume without loss of generality that $j=k$. Hence
$$\vert M_k\vert\geq \frac{1}{k2^{k+2}}\delta'N^{k-1}.$$
Our goal is to apply Lemma \ref{newlemma} to the collection $M_k$ together with the index sets $I_1=\lbrace \ell,\dots,k-1\rbrace$ and $I_2=[\ell-1]\cup\lbrace k\rbrace$. Recall that
$$M_k=\lbrace (x_1,\dots,x_k)\in M^{0}\mid (x_1,\dots,x_{\ell-1},x_k) \text{ is not a bad $([\ell-1]\cup\lbrace k\rbrace)$-tuple}\rbrace.$$
That is, $M_k$ is the collection of all those $(x_1,\dots,x_k)\in X_1\times\dots\times X_k$ with $x_1+\dots+x_k=0$ and such that $(x_1,\dots,x_\ell)$ is a bad $[\ell]$-tuple,  $(x_1,\dots,x_{\ell-1})$ is not a bad $[\ell-1]$-tuple, and
$(x_1,\dots,x_{\ell-1},x_k)$ is not a bad $([\ell-1]\cup\lbrace k\rbrace)$-tuple.

\begin{lem}\label{extensionstomk} For every $(x_\ell,\dots,x_{k-1})\in X_\ell\times\dots\times X_{k-1}$, there are at most $\frac{1}{2}\alpha^{\ell-1}N^{\ell-1}$ different ways to extend $(x_\ell,\dots,x_{k-1})$ to a $k$-cycle $(x_1,\dots,x_k)\in M_k$.
\end{lem}
\begin{proof} First, fix any $x_{\ell }\in X_{\ell }$. Recall that $x_{\ell }$ can be extended to at most
$$\theta\delta'N^{k-2}=\alpha^{k-2}N^{k-2}$$
different $k$-cycles $(x_{1},\dots ,x_{k})\in X_{1}\times \dots
\times X_{k}$. Any bad $[\ell ]$-tuple $(x_{1},\dots ,x_{\ell })$ can
be extended to at least $2\alpha ^{k-\ell -1}N^{k-\ell -1}$ different
$k$-cycles $(x_{1},\dots ,x_{k})$. Hence there are at most
$$\frac{\alpha^{k-2}N^{k-2}}{2\alpha^{k-\ell-1}N^{k-\ell-1}}=\frac{1}{2}\alpha^{\ell-1}N^{\ell-1}$$
ways to extend $x_{\ell }$ to a bad $[\ell ]$-tuple $(x_{1},\dots ,x
_{\ell })$.

Now, fix any $(x_{\ell },\dots ,x_{k-1})\in X_{\ell }\times \dots \times X
_{k-1}$. Each of the at most $\frac{1}{2}\alpha ^{\ell -1}N^{\ell -1}$ extensions of $x_{\ell }$ to a bad $[\ell ]$-tuple $(x_{1},\dots ,x_{\ell })$ gives at most one
possibility for $(x_{1},\dots ,x_{k})\in M_{k}$, because $x_{\ell },
\dots , x_{k-1}$ were fixed and the remaining element $x_{k}$ is already
determined by $x_{1}+\dots +x_{k}=0$.
\end{proof}

Now we will apply Lemma \ref{newlemma} to $M=M_k$ and $r=\frac{1}{k2^{k+2}}$. Indeed, $\vert M_k\vert\geq \frac{1}{k2^{k+2}}\delta'N^{k-1}=r\delta'N^{k-1}$. Let $I_1= \lbrace \ell,\dots, k-1\rbrace$ and $I_2=[\ell-1]\cup \lbrace k\rbrace$. Then $I_1,I_2\su [k]$ are disjoint and $[k]=I_1\cup I_2$. Furthermore $\vert I_1\vert=k-\ell$ and $\vert I_2\vert=\ell$, so since $2\leq \ell\leq k-2$, we have $2\leq \vert I_1\vert\leq k-2$ and $2\leq \vert I_2\vert\leq k-2$. It remains to check (i) and (ii):
\begin{itemize}
\item[(i)] Let $(x_\ell,\dots,x_{k-1})\in X_\ell\times\dots\times X_{k-1}$. By Lemma \ref{extensionstomk} there are at most
$$\frac{1}{2}\alpha^{\ell-1}N^{\ell-1}<2\alpha^{\ell-1}N^{\ell-1}=2\alpha^{k-\vert I_1\vert-1}N^{k-\vert I_1\vert-1}$$
different ways to extend $(x_\ell,\dots,x_{k-1})$ to a $k$-cycle $(x_1,\dots,x_k)\in M_k$.
\item[(ii)] By definition of $M_k$, for any $(x_1,\dots,x_k)\in M_k$, the $([\ell-1]\cup\lbrace k\rbrace)$-tuple $(x_1,\dots,x_{\ell-1},x_k)$ is not a bad $([\ell-1]\cup\lbrace k\rbrace)$-tuple.
\end{itemize}
Thus, all assumptions are satisfied and Lemma \ref{newlemma} yields (as $k\geq 4$)
$$\delta'\theta^{C_{p,k}}\geq \left(\frac{r}{4k}\right)^{C_{p,k}}=\left(\frac{1}{4k^2 2^{k+2}}\right)^{C_{p,k}}=\left(\frac{1}{k^2 2^{k+4}}\right)^{C_{p,k}}\geq \left(\frac{1}{(2^k)^2 2^{k+k}}\right)^{C_{p,k}}=2^{-4kC_{p,k}}.$$
This finishes the proof of Proposition \ref{propo3}.

\section{Proof of Lemma \ref{newlemma}}\label{sect4}

Recall that we are operating under the assumptions of Proposition \ref{propo3}, in particular, $k\geq 4$ and we assume that Theorem \ref{theo2} is true for all smaller values of $k$.

Let $\ell=\vert I_1\vert$, so $2\leq \ell\leq k-2$. Upon relabeling the indices we can assume without loss of generality that $I_1=[\ell]=\lbrace 1,\dots, \ell\rbrace$ and $I_2=\lbrace \ell+1,\dots,k\rbrace$. Note that $\vert I_2\vert=k-\ell$. Furthermore, let
$$Y=\lbrace y\in \Fpn\mid x_1+\dots+x_\ell=-y\text{ has at most }2\alpha^{\ell-1}N^{\ell-1}\text{ solutions with }(x_1,\dots,x_\ell)\in X_1\times\dots\times X_\ell\rbrace.$$
The basic idea of the proof is to consider $(\ell+1)$-cycles in $X_1\times\dots\times X_\ell\times Y$. First, we will show that every $k$-cycle $(x_1,\dots,x_k)\in M$ gives an $(\ell+1)$-cycle $(x_1,\dots,x_\ell,x_{\ell+1}+\dots+x_k)\in X_1\times\dots\times X_\ell\times Y$ by summing up the last $k-\ell$ coordinates. Afterwards, we will prove that there must be a large collection of disjoint $(\ell+1)$-cycles in $X_1\times\dots\times X_\ell\times Y$. We can then apply the hypothesis that Theorem \ref{theo2} holds for $\ell+1<k$, and derive the desired inequality.

\begin{claim}\label{claim31} Let $(x_{\ell+1},\dots,x_k)\in X_{\ell+1}\times\dots\times X_k$ be any $I_2$-tuple. If $x_{\ell+1}+\dots+x_k\not\in Y$, then the $I_2$-tuple $(x_{\ell+1},\dots,x_k)$ is bad.
\end{claim}
\begin{proof} If $x_{\ell+1}+\dots+x_k\not\in Y$, then there are at least $2\alpha^{\ell-1}N^{\ell-1}$ choices for $(x_1,\dots,x_\ell)\in X_1\times\dots\times X_\ell$ such that 
$x_1+\dots+x_\ell=-(x_{\ell+1}+\dots+x_k)$. For each such choice we obtain a $k$-cycle $(x_1,\dots,x_k)\in X_1\times \dots\times X_k$ extending the $I_2$-tuple $(x_{\ell+1},\dots,x_k)$. Hence the $I_2$-tuple $(x_{\ell+1},\dots,x_k)$ is bad.
\end{proof}

\begin{claim}\label{claim32} For every $(x_1,\dots,x_k)\in M$ we have $x_{\ell+1}+\dots+x_k\in Y$.\end{claim}
\begin{proof}This follows immediately from Claim \ref{claim31} and assumption (ii).\end{proof}

Thus, every $k$-cycle $(x_1,\dots,x_k)\in M$ gives an $(\ell+1)$-cycle $(x_1,\dots,x_\ell,x_{\ell+1}+\dots+x_k)\in X_1\times\dots\times X_\ell\times Y$. The following claim will be a useful tool for proving that there is a large collection of disjoint $(\ell+1)$-cycles in $X_1\times\dots\times X_\ell\times Y$.

\begin{claim}\label{claim33} For every $y\in Y$, the number of $k$-cycles $(x_1,\dots,x_k)\in M$ with $x_{\ell+1}+\dots+x_k=y$ is at most $4\theta\delta'N^{k-2}$.\end{claim}
\begin{proof}Note that for any such $k$-cycle $(x_1,\dots,x_k)\in M$ we have
$$x_1+\dots+x_\ell=-(x_{\ell+1}+\dots+x_k)=-y.$$
As $y\in Y$ there are at most $2\alpha^{\ell-1}N^{\ell-1}$ choices for $(x_1,\dots,x_\ell)\in X_1\times\dots\times X_\ell$ with $x_1+\dots+x_\ell=-y$. For each such choice for $(x_1,\dots,x_\ell)$, by assumption (i) there are at most $2\alpha^{k-\ell-1}N^{k-\ell-1}$ ways to extend $(x_1,\dots,x_\ell)$ to a $k$-cycle $(x_1,\dots,x_k)\in M$. So all in all there are at most
$$2\alpha^{\ell-1}N^{\ell-1}\cdot 2\alpha^{k-\ell-1}N^{k-\ell-1}=4\alpha^{k-2}N^{k-2}=4\theta\delta'N^{k-2}$$
choices for $(x_1,\dots,x_k)\in M$ with $x_{\ell+1}+\dots+x_k=y$.
\end{proof}

We will now prove that there is a large collection of disjoint $(\ell+1)$-cycles in $X_1\times\dots\times X_\ell\times Y$. Let us choose a maximal collection of disjoint $(\ell+1)$-cycles in $X_1\times\dots\times X_\ell\times Y$, and let the number of $(\ell+1)$-cycles in our collection be $t$. Furthermore, let $X_1'\su X_1$, \dots, $X_\ell'\su X_\ell$ and $Y'\su Y$ consist of the elements involved in the $t$ disjoint $(\ell+1)$-cycles, then $\vert X_1'\vert=\vert X_2'\vert=\dots=\vert X_\ell'\vert=\vert Y'\vert=t$. Since the collection is maximal, every $(\ell+1)$-cycle $(x_1,\dots,x_\ell,y)\in X_1\times\dots\times X_\ell\times Y$ satisfies $y\in Y'$ or $x_i\in X_i'$ for some $i\in \lbrace 1,\dots,\ell\rbrace$.

\begin{claim}\label{t-value} We have $t\geq \frac{r}{2k\theta}N$.
\end{claim}
\begin{proof} Recall that for each $k$-cycle $(x_1,\dots,x_k)\in M$ we obtain an $(\ell+1)$-cycle $(x_1,\dots,x_\ell,y)\in X_1\times\dots\times X_\ell\times Y$ by taking $y=x_{\ell+1}+\dots+x_k\in Y$ (see Claim \ref{claim32}). Hence for each $k$-cycle $(x_1,\dots,x_{k})\in M$ we must have $x_{\ell+1}+\dots+x_{k}\in Y'$ or $x_i\in X_i'$ for some $i\in \lbrace 1,\dots,\ell\rbrace$. Note that by the assumptions of Proposition \ref{propo3}, for each $i=1,\dots,\ell$, at most $\vert X_i'\vert\cdot \theta\delta'N^{k-2}=t\theta\delta'N^{k-2}$ different $k$-cycles $(x_1,\dots,x_{k})\in M$ can satisfy $x_i\in X_i'$. Furthermore, for each $y\in Y'$ we have $y\in Y$ and therefore by Claim \ref{claim33} there are at most $4\theta\delta'N^{k-2}$ different $k$-cycles $(x_1,\dots,x_k)\in M$ with $x_{\ell+1}+\dots+x_{k}=y$. Thus, at most $\vert Y'\vert\cdot 4\theta\delta'N^{k-2}=4t\theta\delta'N^{k-2}$ different $k$-cycles $(x_1,\dots,x_{k})\in M$ satisfy $x_{\ell+1}+\dots+x_{k}\in Y'$. All in all we obtain, using $\ell\leq k-2$ and $k\geq 4$,
\[r\delta'N^{k-1}\leq \vert M\vert\leq \ell\cdot t\theta\delta'N^{k-2}+4t\theta\delta'N^{k-2}=(\ell+4)t\theta\delta'N^{k-2}\leq 2kt\theta\delta'N^{k-2}.\]
Thus, indeed $t\geq \frac{r}{2k\theta}N$.
\end{proof}

We assumed that Theorem \ref{theo2} holds for $\ell+1<k$ and we found a collection of at least $t$ disjoint $(\ell+1)$-cycles in $X_1'\times \dots\times X_\ell'\times Y'$. So by Theorem \ref{theo2} for $\ell+1$, the total number of $(\ell+1)$-cycles in $X_1'\times \dots\times X_\ell'\times Y'$ is at least
$$\left(\frac{t}{N}\right)^{C_{p,\ell+1}} N^{\ell}.$$

On the other hand, by the definition of $Y$, for each $y\in Y'\su Y$ there are at most $2\alpha^{\ell-1}N^{\ell-1}$ different solutions for $x_1+\dots+x_\ell=-y$ with $(x_1,\dots,x_\ell)\in X_1\times\dots\times X_\ell$. In particular, each $y\in Y'$ can be extended to at most $2\alpha^{\ell-1}N^{\ell-1}$ different $(\ell+1)$-cycles $(x_1,\dots,x_\ell,y)\in X_1'\times\dots\times X_\ell'\times Y'$. Thus, the total number of $(\ell+1)$-cycles in $X_1'\times \dots\times X_\ell'\times Y'$ is at most
$$\vert Y'\vert \cdot 2\alpha^{\ell-1}N^{\ell-1}=2t\alpha^{\ell-1}N^{\ell-1}.$$
So we obtain
$$\left(\frac{t}{N}\right)^{C_{p,\ell+1}} N^{\ell}\leq 2t\alpha^{\ell-1}N^{\ell-1},$$
and therefore together with Claim \ref{t-value}
$$2\alpha^{\ell-1}\geq \left(\frac{t}{N}\right)^{C_{p,\ell+1}-1}\geq \left(\frac{r}{2k\theta}\right)^{C_{p,\ell+1}-1}.$$
By $C_{p,3}>2$ and $\ell\geq 2$ we have $C_{p,\ell+1}=(\ell-1)(C_{p,3}-1)+1>2$. Thus, we obtain
$$\alpha^{\ell-1}\geq \frac{1}{2}\left(\frac{r}{2k\theta}\right)^{C_{p,\ell+1}-1}\geq \left(\frac{r}{4k\theta}\right)^{C_{p,\ell+1}-1}=\left(\frac{r}{4k\theta}\right)^{(\ell-1)(C_{p,3}-1)}.$$
Taking this to the $\frac{k-2}{\ell-1}$-th power gives
$$\theta\delta'=\alpha^{k-2}\geq \left(\frac{r}{4k\theta}\right)^{(k-2)(C_{p,3}-1)}=\left(\frac{r}{4k\theta}\right)^{C_{p,k}-1}.$$
Now rearranging yields
$$\delta'\theta^{C_{p,k}}\geq \left(\frac{r}{4k}\right)^{C_{p,k}-1}\geq \left(\frac{r}{4k}\right)^{C_{p,k}},$$
as desired. This finishes the proof of Lemma \ref{newlemma}.

\section{Proof of Theorem \ref{theo1} and Theorem \ref{theo2}}\label{sect5}

We next prove Theorem \ref{theo2} by induction on $k$. The base case $k=3$ is \cite[Theorem 3]{foxlovasz}. Let us therefore fix some $k\geq 4$ and assume that we have proved Theorem \ref{theo2} for all smaller values of $k$. Let us also fix the prime $p$, and to simplify notation, set $C=C_{p,k}$.
For the fixed value of $k$, we 
will first prove a version of Theorem \ref{theo2} with a slightly weaker bound, see Proposition \ref{theo2-approx} below. Afterwards, we will use a power trick to obtain the actual statement of Theorem \ref{theo2}.

Fix a sufficiently small real number $0<t_{p,k}<\frac{1}{2}$ such that
$$t^{C}(\log(1/t))^{2C}< 2^{-6kC}C^{-2C}$$
for all $0<t<t_{p,k}$ and such that the function $t(\log(1/t))^{2C}$ is monotonically increasing on the interval $(0,t_{p,k})$.

\begin{propo}\label{theo2-approx} Let $0<\eps<t_{p,k}$ and $X_1,\dots,X_k\su \Fpn$ be such that there is a collection of at least $\eps N$ disjoint $k$-cycles in $X_1\times \dots\times X_k$. Then the total number of $k$-cycles in $X_1\times \dots\times X_k$ is at least
$$2^{-6kC}C^{-2C}\frac{\eps^{C}}{\log(1/\eps)^{2C}}N^{k-1}.$$
\end{propo}
\begin{proof} Set
\begin{equation}\label{eqdeftstar}
t^{*}=2^{-6kC}C^{-2C}\frac{\eps^{C}}{\log(1/\eps)^{2C}}.
\end{equation}
Suppose for contradiction that the number of $k$-cycles in $X_1\times \dots\times X_k$ is less than $t^{*}N^{k-1}$. Note that we have $t^{*}<\eps<t_{p,k}<\frac{1}{2}$ as $C>1$.

Our goal is to apply Proposition \ref{propo3}. In order to do so, we will step by step delete points one at a time from  the sets $X_1$,\dots,$X_k$ that are contained in too many $k$-cycles. At every moment during this procedure, let $\delta'$ be such that the total number of $k$-cycles in $X_1\times\dots\times X_k$ is $\delta'N^{k-1}$ (note that $\delta'$ changes during the procedure as points get deleted). Whenever for some $i=1,\dots,k$ there is a point of $X_i$ that occurs as $x_i$ in at least 
$$\frac{8(\log(1/\delta'))^{2}}{\eps}\delta'N^{k-2}$$
different $k$-cycles $(x_1,\dots,x_k)\in X_1\times \dots\times X_k$, delete this point from the corresponding set $X_i$ and update $\delta'$ (such that again the number of $k$-cycles in $X_1\times\dots\times X_k$ is equal to $\delta'N^{k-1}$). Note that during this process, as points get deleted, the number of $k$-cycles in $X_1\times\dots\times X_k$ decreases. Hence $\delta'$ is decreasing during the process. In particular we always have $\delta'<t^{*}<\frac{1}{2}$. 

Note that for any positive integer $j$, if $2^{-(j+1)}\leq \delta'\leq 2^{-j}$ at the beginning of a deletion step, then the deleted point is contained in at least
$$\frac{8(\log(1/\delta'))^{2}}{\eps}\delta'N^{k-2}\geq \frac{8j^{2}}{\eps}2^{-(j+1)}N^{k-2}=\frac{4j^{2}}{\eps}2^{-j}N^{k-2}$$
$k$-cycles in $X_1\times\dots\times X_k$. So as long as $2^{-(j+1)}\leq \delta'\leq 2^{-j}$, then in each step at least $\frac{4j^{2}}{\eps}2^{-j}N^{k-2}$ different $k$-cycles in $X_1\times\dots\times X_k$ get destroyed. However, if $2^{-(j+1)}\leq \delta'\leq 2^{-j}$, then there are only $\delta'N^{k-1}\leq 2^{-j}N^{k-1}$ different $k$-cycles. Hence the number of steps where at the beginning of the step we have $2^{-(j+1)}\leq \delta'\leq 2^{-j}$ is at most
$$\frac{2^{-j}N^{k-1}}{\frac{4j^{2}}{\eps}2^{-j}N^{k-2}}=\frac{\eps}{4j^2}N.$$
So the total number of steps in the process above is at most
$$\sum_{j=1}^{\infty}\frac{\eps}{4j^2}N=\frac{\eps}{4}N\sum_{j=1}^{\infty}\frac{1}{j^2}\leq \frac{\eps}{2}N.$$
In particular, the number of points deleted before the process terminates is at most $\frac{\eps}{2}N$.

In the beginning, we had at least $\eps N$ disjoint $k$-cycles in $X_1\times\dots\times X_k$. Since at most $\frac{\eps}{2}N$ points were deleted, after the deletion process we still have at least $\frac{\eps}{2}N$ disjoint $k$-cycles in $X_1\times\dots\times X_k$. In particular we have $0<\delta'<t^{*}$ after the deletion process.

Now let
$$\theta=\frac{8(\log(1/\delta'))^{2}}{\eps}>1.$$
When the deletion process has terminated, we have $\delta' N^{k-1}$ different $k$-cycles in $X_1\times\dots\times X_k$ (with $\delta'>0$) and for every $i=1,\dots,k$, each point of $X_i$ occurs as $x_i$ in at most $\theta\delta'N^{k-2}$ different $k$-cycles $(x_1,\dots,x_k)\in X_1\times \dots\times X_k$.

So we can apply Proposition \ref{propo3} to the sets $X_1,\dots, X_k\su \Fpn$ after the deletion process and we obtain
$$\delta'\left(\frac{8(\log(1/\delta'))^{2}}{\eps}\right)^{C}=\delta'\theta^{C_{p,k}}\geq 2^{-4kC_{p,k}}=2^{-4kC}.$$
Hence
$$\delta'(\log(1/\delta'))^{2C}\geq 2^{-4kC}8^{-C}\eps^{C}= 2^{-4kC}2^{-3C}\eps^{C}\geq 2^{-5kC}\eps^{C}.$$
Since $0<\delta'<t^{*}<t_{p,k}$ and the function $t(\log(1/t))^{2C}$ is monotonically increasing on the interval $(0,t_{p,k})$, this implies that
\begin{equation}\label{eqtstar}
t^{*}(\log(1/t^{*}))^{2C}\geq 2^{-5kC}\eps^{C}.
\end{equation}

On the other hand, note that $\eps<t_{p,k}$ implies by the choice of $t_{p,k}$ that
$$\eps^{C}(\log(1/\eps))^{2C}< 2^{-6kC}C^{-2C}.$$
Together with \eqref{eqdeftstar}, we obtain $t^{*}>\eps^{C}\cdot \eps^{C}=\eps^{2C}$, and consequently $\log(1/t^{*})<\log(1/\eps^{2C})=2C\log(1/\eps)$. Using \eqref{eqdeftstar} again, we therefore have
$$t^{*}(\log(1/t^{*}))^{2C}<2^{-6kC}C^{-2C}\frac{\eps^{C}}{\log(1/\eps)^{2C}}\cdot (2C\log(1/\eps))^{2C}=
2^{-6kC}2^{2C}\eps^{C}<2^{-5kC}\eps^{C}.$$
This contradicts \eqref{eqtstar}.
Hence our assumption must have been wrong and the number of $k$-cycles in $X_1\times \dots\times X_k$ is indeed at least
$$t^{*}N^{k-1}=2^{-6kC}C^{-2C}\frac{\eps^{C}}{\log(1/\eps)^{2C}}N^{k-1}.$$
This finishes the proof of Proposition \ref{theo2-approx}.
\end{proof}

We now complete the induction step by deducing Theorem \ref{theo2} for the fixed value of $k$ from Proposition \ref{theo2-approx}.

\begin{proof}[Proof of Theorem \ref{theo2} for $k$] Let $0<\eps<1$ and let  $X_1,\dots,X_k\su \Fpn$ be such that there is a collection of at least $\eps N$ disjoint $k$-cycles in $X_1\times \dots\times X_k$. Let the total number of $k$-cycles $(x_1,\dots,x_k)\in X_1\times \dots\times X_k$ be $\delta' N^{k-1}$. We need to show that $\delta'\geq \eps^{C}=\eps^{C_{p,k}}=\delta$.

Let $(x_1^{(1)},\dots,x_k^{(1)}),(x_1^{(2)},\dots,x_k^{(2)}),\dots, (x_1^{(h)},\dots,x_k^{(h)})$ be disjoint $k$-cycles in $X_1\times \dots\times X_k$ for some $h\geq \eps N$. Note that then for each $i=1,\dots,k$ the points $x_i^{(1)},\dots, x_i^{(h)}$ are distinct.

For each positive integer $m$ consider $\Fp^{nm}=\Fpn\times\dots\times \Fpn$ and for $i=1,\dots,k$ let $X_i^{m}\su \Fp^{nm}$ be the subset given by $X_i\times\dots\times X_i$. Note that a $k$-cycle in $X_1^{m}\times\dots\times X_k^{m}$ corresponds to a $k$-cycle in $X_1\times \dots\times X_k$ in the first $n$ coordinates, a $k$-cycle in $X_1\times \dots\times X_k$ in the next $n$ coordinates and so on. So a $k$-cycle in $X_1^{m}\times\dots\times X_k^{m}$ is the same as an $m$-tuple of $k$-cycles in $X_1\times \dots\times X_k$. Thus, the number of $k$-cycles in $X_1^{m}\times\dots\times X_k^{m}$ is equal to
$$(\delta' N^{k-1})^m=\delta'^mp^{nm(k-1)}.$$

For any $j_1,\dots,j_m\in [h]$, the points $(x_1^{(j_1)},x_1^{(j_2)}\dots,x_1^{(j_m)})$, $(x_2^{(j_1)},x_2^{(j_2)}, \dots,x_2^{(j_m)})$, \dots,$(x_k^{(j_1)},x_k^{(j_2)}\dots,x_k^{(j_m)})$ form a $k$-cycle in $X_1^{m}\times\dots\times X_k^{m}$. Also note that for distinct tuples $(j_1,\dots,j_m)\in [h]^{m}$ these $k$-cycles in $X_1^{m}\times\dots\times X_k^{m}$ are disjoint. So we have found a collection of $h^{m}\geq \eps^{m}N^m=\eps^{m}p^{nm}$ disjoint $k$-cycles in $X_1^{m}\times\dots\times X_k^{m}$.

For $m$ sufficiently large we have $\eps^{m}<t_{p,k}$, so we can apply Proposition \ref{theo2-approx} to $X_1^{m},\dots, X_k^{m}\su \Fp^{nm}$ and obtain that there are at least
$$2^{-6kC}C^{-2C}\frac{(\eps^m)^{C}}{\log(1/\eps^m)^{2C}}(p^{nm})^{k-1}$$
$k$-cycles in $X_1^{m}\times\dots\times X_k^{m}$. On the other hand, we have seen above that the number of $k$-cycles in $X_1^{m}\times\dots\times X_k^{m}$ equals $\delta'^mp^{nm(k-1)}$. Hence for all sufficiently large $m$ we must have
$$\delta'^m\geq 2^{-6kC}C^{-2C}\frac{(\eps^m)^{C}}{m^{2C}\log(1/\eps)^{2C}}.$$
Taking $m$-th roots on both sides gives
$$\delta'\geq 2^{-6kC/m}C^{-2C/m}\frac{\eps^{C}}{m^{2C/m}\log(1/\eps)^{2C/m}}$$
for all sufficiently large $m$. For $m\to \infty$ the right-hand side of the last inequality tends to $\eps^{C}$. Thus,
$$\delta'\geq \eps^{C}= \eps^{C_{p,k}}$$
as desired.
\end{proof}

This finishes the induction step. Thus, Theorem \ref{theo2} is proved for all $k\geq 3$.

\begin{remark} Since $C_{p,k}$ might not be the optimal exponent in Theorem \ref{theo1}, the reader might wonder why we concern ourselves with using the power trick to remove the logarithmic terms from the bound in Proposition \ref{theo2-approx}. However, note that Theorem \ref{theo2} does not only have a slightly better bound than Proposition \ref{theo2-approx} (namely by removing the logarithmic terms), but also Proposition \ref{theo2-approx} was only stated and proved for sufficiently small $\eps$. With the power trick, we obtain Theorem \ref{theo2} for all $\eps$. This makes the argument in the inductive proof much cleaner.\end{remark}

Finally, let us deduce Theorem \ref{theo1} from Theorem \ref{theo2}.

\begin{proof}[Proof of Theorem \ref{theo1}] Let $X_1,\dots, X_k\su \Fpn$ and choose a maximal collection of disjoint $k$-cycles in $X_1\times\dots\times X_k$. If there are at least $\eps N$ disjoint $k$-cycles in our collection, then by Theorem \ref{theo2} the total number of $k$-cycles $(x_1,\dots,x_k)\in X_1\times \dots\times X_k$ is at least $\delta N^{k-1}$. Otherwise, the collection consists of less than $\eps N$ disjoint $k$-cycles. Then let us delete all points of these less than $\eps N$ disjoint $k$-cycles from the corresponding sets $X_i$. Then from each $X_i$ we will have deleted less than $\eps N$ elements and no $k$-cycles remain in $X_1\times\dots\times X_k$, because the collection of disjoint $k$-cycles we considered in the beginning was maximal.
\end{proof}

For any $X_1,\ldots,X_k\su \Fpn$, the maximum number of disjoint $k$-cycles is within a factor $k$ of the number of elements one needs to delete from each set in order to remove all $k$-cycles. It follows that Theorem \ref{theo1} applied to $\eps/ k$ also implies Theorem \ref{theo2} with $\delta=(\eps/ k)^{C_{p,k}}$. Hence Theorem \ref{theo1} and Theorem \ref{theo2} are equivalent up to a change of a constant factor in the value of $\delta$.

\textit{Acknowledgments.} The authors would like to thank the anonymous referees for their helpful comments and suggestions.


\begin{thebibliography}{99}

\bibitem{BHATTACHARYYA13}
A.~Bhattacharyya, \emph{Guest column: {O}n testing affine-invariant properties
	over finite fields}, ACM SIGACT News \textbf{44} (2013), 53--72.

\bibitem{BCSX10}
A.~Bhattacharyya, V.~Chen, M.~Sudan, and N.~Xie, \emph{Testing linear-invariant
	non-linear properties: A short report}, Property Testing: Current Research
and Surveys (O.~Goldreich, ed.), Springer Berlin Heidelberg, Berlin,
Heidelberg, 2010, 260--268.


\bibitem{BGRS12}
A.~Bhattacharyya, E.~Grigorescu, P.~Raghavendra, and A.~Shapira, \emph{Testing
	odd-cycle-freeness in {B}oolean functions}, Combin. Probab. Comput.
\textbf{21} (2012), 835--855.


\bibitem{BGS15}
A.~Bhattacharyya, E.~Grigorescu, and A.~Shapira, \emph{A unified framework for
	testing linear-invariant properties}, Random Structures Algorithms
\textbf{46} (2015), 232--260.

\bibitem{bx} A.~Bhattacharyya and N.~Xie, \emph{Lower Bounds for Testing Triangle-freeness in Boolean Functions}, Comput. Complexity \textbf{24} (2015), 65--101. A preliminary version appeared in SODA 2010, 87--98.

\bibitem{blasiak-et-al} 
J.~{Blasiak}, T.~{Church}, H.~{Cohn}, J.~A.~{Grochow}, E.~Naslund, W.~F.~Sawin,
and C.~{Umans}, \emph{{On cap sets and the group-theoretic approach to matrix multiplication}}, Discrete Anal. 2017, Paper No. 3, 27pp.

\bibitem{crootlevpach} E.~Croot, V.~F.~Lev, and P.~P.~Pach, \emph{Progression-free sets in $\mathbb{Z}_4^n$ are exponentially small}, Ann. of Math. (2) \textbf{185} (2017), 331--337.

\bibitem{ellengijs} J.~S.~Ellenberg and D.~Gijswijt, \emph{On large subsets of $\mathbb{F}_q^n$ with no three-term arithmetic progression}, Ann. of Math. (2) \textbf{185} (2017), 339--343.

\bibitem{fox} J.~Fox, \emph{A new proof of the graph removal lemma}, Ann. of Math. (2) \textbf{174} (2011), 561--579.

\bibitem{foxlovasz}
  J.~Fox and L.~M.~Lov\'{a}sz, \emph{A tight bound for Green's arithmetic triangle removal lemma in vector spaces}, Adv. Math. \textbf{321} (2017), 287--297.

\bibitem{green} B.~Green, \emph{A Szemer\'{e}di-type regularity lemma in abelian groups, with applications}, Geom. Funct. Anal. \textbf{15} (2005), 340--376.

\bibitem{HX15}
I.~Haviv and N.~Xie, \emph{Sunflowers and testing triangle-freeness of
	functions}, Comput. Complexity  \textbf{26}  (2017), 497--530. A preliminary version appeared in 
I{TCS}'15, 2015, 357--366.

\bibitem{KSS17}
R.~Kleinberg, W.~F.~Sawin, and D.~E.~Speyer, \emph{{The growth rate of
		tri-colored sum-free sets}}, preprint, 2017, arXiv:1607.00047.

\bibitem{kral-serra-vena} D.~Kr\'{a}l, O.~Serra, and L.~Vena, \emph{A combinatorial proof for the removal lemma for groups}, J. Combin. Theory Ser. A \textbf{116} (2009), 971--978.

\bibitem{lovaszsauermann}
L.~M.~Lov\'{a}sz and L.~Sauermann, \emph{{A lower bound for the $k$-multicolored sum-free problem in $\mathbb{Z}_m^n$}}, preprint, 2018,  	arXiv:1804.08837.

\bibitem{tao} T.~Tao, \emph{A symmetric formulation of the Croot–Lev–Pach–Ellenberg–Gijswijt capset bound}, blog post, 2016, \texttt{http://terrytao.wordpress.com/2016/05/18/a}.
\end{thebibliography}
\end{document}